\documentclass[english]{amsart}%
\usepackage{amsfonts}
\usepackage{amsmath,color}
\usepackage{amssymb}
\usepackage{amsthm}
\usepackage{amscd}
\usepackage{babel}
\usepackage[latin1]{inputenc}
\usepackage{graphicx}
\usepackage{amsmath}%
\providecommand{\U}[1]{\protect\rule{.1in}{.1in}}
\newtheorem{teor}{Theorem}
\newtheorem{cor}{Corollary}
\newtheorem{prop}{Proposition}
\newtheorem{con}{Conjecture}
\newtheorem{lem}{Lemma}

\theoremstyle{definition}
\newtheorem{defi}{Definition}
\newtheorem*{rem}{Remark}

\renewcommand{\subjclassname}{AMS \textup{2010} Mathematics Subject
Classification\ }

\begin{document}

\author{L. Bayón}
\address{Departamento de Matemáticas, Universidad de Oviedo\\
Avda. Calvo Sotelo s/n, 33007 Oviedo, Spain}
\email{bayon@uniovi.es}
\author{P. Fortuny}
\address{Departamento de Matemáticas, Universidad de Oviedo\\
Avda. Calvo Sotelo s/n, 33007 Oviedo, Spain}
\email{fortunypedro@uniovi.es}
\author{J. Grau}
\address{Departamento de Matemáticas, Universidad de Oviedo\\
Avda. Calvo Sotelo s/n, 33007 Oviedo, Spain}
\email{grau@uniovi.es}
\author{A. M. Oller-Marcén}
\address{Centro Universitario de la Defensa de Zaragoza - IUMA\\
Ctra. Huesca s/n, 50090 Zaragoza, Spain}
\email{oller@unizar.es}
\author{M. M. Ruiz}
\address{Departamento de Matemáticas, Universidad de Oviedo\\
Avda. Calvo Sotelo s/n, 33007 Oviedo, Spain}
\email{mruiz@uniovi.es}

\title[Randomized Best-or-Worst and Postdoc problems]{The Best-or-Worst and the Postdoc problems with random number of candidates}

\begin{abstract}
In this paper we consider two variants of the Secretary problem: The
Best-or-Worst and the Postdoc problems. We extend previous work by considering
that the number of objects is not known and follows either a discrete Uniform
distribution $\mathcal{U}[1,n]$ or a Poisson distribution $\mathcal{P}%
(\lambda)$. We show that in any case the optimal strategy is a threshold
strategy, we provide the optimal cutoff values and the asymptotic
probabilities of success. We also put our results in relation with closely related work.

\end{abstract}
\maketitle

\subjclassname{60G40, 62L15}

\section{Introduction}

The classical \emph{Secretary problem} has been extensively studied in the
fields of applied probability, statistics or decision theory and has been
considered by many authors (see \cite{FER,FER2,2009} for an extensive
bibliography). It can also be posed as a decision making problem in a game
with the following rules:

\begin{enumerate}
\item We have to choose one object from a set.

\item The total number of objects in the set is known.

\item The objects are rankable from best to worst.

\item The objects appear sequentially and in random order.

\item Each object must be accepted or rejected before the next one appears.

\item The decision depends only on the relative ranks of the objects examined
so far.

\item Rejected objects cannot be called back.

\item We want to maximize the probability of selecting the best object.
\end{enumerate}

Dynkin \cite{48} and Lindley \cite{101} independently proved that, in the
previous setting, the best strategy consists in observing roughly $n/e$ of the
objects and then choosing the first one that is better than all those observed
so far. This strategy returns the best object with a probability of at least
$1/e$, this being its approximate value for large values of $n$. This
well-known solution was later refined by Gilbert and Mosteller \cite{gil},
showing that $\left\lfloor (n-\frac{1}{2})e^{-1}+\frac{1}{2}\right\rfloor $ is
a better approximation than $\lfloor n/e\rfloor$, although the difference is
never greater than 1.

We mention here that the classical Secretary problem is just a special case of the problem of stopping without recall on the very last {\it interesting} event, since it suffices to define interesting as {\it better than the previous ones}. The solution is therefore a corollary of the odds-theorem of optimal stopping \cite{BR3}. Moreover, Bruss \cite{BR4} shows that the lower bound of $1/e$ for the success probability holds, remarkably, in all generality for whatever law of interesting events. For further developments se also \cite{DEN1,DEN2,LOU}.

If we modify rule (8) above, we can get variants of the secretary problem.
Some of them also have simple, elegant solutions. For example, if we consider

\begin{enumerate}
\item[(8$^{\prime}$)] We want to maximize the probability of selecting the
second best object.
\end{enumerate}

we obtain the so-called \emph{Postdoc problem}  in \cite{postdoc}. In this setting the probability of success for an even number of applicants is exactly
$\frac{n}{4(n-1)}$. This probability tends to 1/4 as $n$ tends to infinity,
illustrating the fact that it is easier to pick the best than the second best. This variant was also considered in \cite{bayon, rose, aesima}.

On the other hand, if we consider

\begin{enumerate}
\item[(8$^{\prime\prime}$)] We want to maximize the probability of selecting
either the best or the worst object.
\end{enumerate}

we get the so-called \emph{Best-or-Worst problem}. This variant can be found
on \cite{fergu} as a multicriteria problem in the perfect negative dependence
case.  In \cite{bayon} we considered the Best-or-Worst and the Postdoc problems
proving that both of them share the same threshold strategy as optimal
stopping rule and that the probability of success in the Best-or-Worst problem
is twice the probability of success in the Postdoc problem.

Besides these, many other variants of the classical Secretary problem have been recently studied, specially in the framework of partially ordered objects \cite{poset2, poset1, garrod} or matroids \cite{baba, matro,soto}.

Interesting lines of work also arise if we also modify rule (2) above. If the
number of objects is unknown, the decision maker faces an additional risk
because if he rejects an object, he may then discover that it was the last
one, in which case he fails. In \cite{sonin} the case in which the number of
objects follows a discrete Uniform distribution $\mathcal{U}[1,n]$ was studied
for the classical secretary problem. In this setting, the cutoff value for
large $N$ is approximately $Ne^{-2}$ and the probability of success is
$2e^{-2}$. This same paper also dealt with the case in which the number of
candidates follows a Poisson distribution $\mathcal{P}(\lambda)$ showing that
the optimal stopping limit relation is $r^{\ast}(\lambda)/\lambda\rightarrow
e^{-1}$ and that this is also the asymptotic value of the probability of success. This was first studied in a continuous time setting by Cowan and Zabczyk \cite{COW} for a Poisson process of candidates with known arrival rate and then generalized by Bruss \cite{BR2} for an unknown arrival rate. See also Szajowski's work \cite{SZA} for a corresponding game version.

In the present paper, we want to extend the work done in \cite{bayon} by
considering the Best-or-Worst and the Postdoc problems when the number of
objects follows a Uniform distribution $\mathcal{U}[1,n]$ or a Poisson
distribution $\mathcal{P}(\lambda)$.

The paper is organized as follows: In Section 2, we recall the relation between the Best-or-Worst and the Postdoc problems for a known number of objects and extend it to our setting. In Section 3 we show that in the considered situations; i.e., if the random number of candidates follows either
a discrete Uniform distribution $\mathcal{U}[1,n]$ or a Poisson distribution
$\mathcal{P}(\lambda)$ the optimal strategy is still a threshold strategy.
After that, sections 4 and 5 deal with the Uniform case and with the Poisson
case, respectively. Finally, Section 6 presents a comparative table of the
results and some concluding remarks.

\section{The relation between the Best-or-Worst problem and the Postdoc problem}

The following theorem (see \cite[Section 4]{bayon}) establishes the relation between the optimal strategies in the Best-or-Worst problem and the Postdoc problem when the number of objects in known.

\begin{teor}
\label{Tcon} Let us define a nice candidate as an object which is either better
or worse than all the preceding ones (in the Best-or-Worst problem) or which
is the second better than all the preceding ones (in the Postdoc problem).
Then, if $n$ is the total number of objects, the following strategy is optimal:
\begin{enumerate}
\item Reject the $\lfloor\frac{n}{2} \rfloor$ first inspected objects
regardless their rank.
\item After that, accept the first nice candidate.
\end{enumerate}
Moreover, if $P_{BW}(n)$ and $P_{PD}(n)$ are the probabilities of success
following this strategy in the Best-or-Worst and in the Postdoc problem
respectively, then we have that
\[
2P_{PD}(n)=P_{BW}(n)=%
\begin{cases}
\frac{n}{2(n-1)}, & \text{if $n$ is even};\\
\frac{n+1}{2n}, & \text{if $n$ is odd}.
\end{cases}
\]
\end{teor}
In \cite {bayon} it was also shown that, in the Postdoc problem, selecting a candidate that is better than all the previous ones has the same probability of success as waiting for the next nice candidate (second better than the previous ones). This means that the optimal strategy can neglect if a given candidate is better than all the preceding ones and focus only on whether the candidate is the second better than all the preceding ones.

Note that Theorem \ref{Tcon} implies that, when the number of objects is known, both problems share the same optimal threshold strategy. Moreover, under this strategy the probabilities of success in both problems are closely related (one is twice the other). We will now see that, as long as we follow a threshold strategy, this relationship still holds even if the number of objects is unknown. To do so, we first need two easy results.

\begin{prop}
If $n$ is the total number of objects, let $A_{n}^{BW}(r)$ and $ A_{n}^{PD}(r)$ denote the probability of success if we accept a nice candidate at the $r$-th step in the Best-or-Worst and in the Postdoc problem, respectively. Then,
\begin{align*}
A_{n}^{BW}(r)&=\frac{r}{n}\\
A_{n}^{PD}(r)&=\frac{r (r-1)}{n(n-1)}.
\end{align*}
\end{prop}

Recall that a threshold strategy with cutoff value $r$ consists of rejecting any inspected object before the $r$-th inspection and then accepting the first nice candidate after that. The following result is a direct consequence of the previous proposition.

\begin{prop} \label{umbral}
If $n>1$ is the total number of objects, let $R_{n}^{BW}(r)$ and $R_{n}^{PD}(r)$ denote the probability of success following a threshold strategy with cutoff value $r$ in the Best-or-Worst and in the Postdoc problem, respectively. Then,
\[
2R^{PD}_{n}(r)=R^{BW}_{n}(r)=
\begin{cases}
\displaystyle{2/n}, & \text{if $r=0$   };\\
\displaystyle{\frac{2r(n-r)}{n(n-1)}}, & \text{if $ n\geq r$  };\\
0, & \text{if $n<r$ }.
\end{cases}
\]
\end{prop}

Now, we can extend part of Theorem \ref{Tcon} to the case when the number of objects is unknown.

\begin{cor}
\label{corred}
If $X$ is the random variable defining the number of objects, let $P_X^{BW}(r)$ and $P_X^{PD}(r)$ denote the probability of success following a threshold strategy with cutoff value $r>1$ in the Best-or-Worst and in the Postdoc problem, respectively. Then,
$$P_X^{BW}(r)=2P_X^{PD}(r).$$
\end{cor}
\begin{proof}
Taking into account the previous proposition, it is enough to observe that
$$P_X^{BW}(r) = \sum_{i\geq r+1}R^{BW}_{i}(r)\cdot p(X=i)=\sum_{i\geq r+1}2R^{PD}_{i}(r)\cdot p(X=i)=2P_X^{PD}(r).$$
\end{proof}

This corollary will be important in the sequel because, once we show that the optimal strategy is a threshold strategy and regardless the distribution followed by the unknown number of objects, both problems will share the same optimal cutoff value and the probability of success in the Best-or-Worst problem will be twice the probability of success in the Postdoc problem. Consequently, we will be able to focus on just one of them, namely the Best-or-Worst problem.

\section{Threshold strategies for a random number of objects}

As we have already mentioned, when the number of objects is known, the optimal
strategy is a threshold strategy. Unfortunately, this is not necessarily the
case if the number of objects is random. For example, let us assume that in
the classical secretary problem the number of objects is a discrete random
variable $X$ such that $p(X=100)=0.99$ and $p(X=1000)=0.01$. Clearly, in such
a situation the optimal strategy is not a threshold strategy. In fact, if at
the 100-th step we inspect an object which is better than all the preceding
ones, it must be accepted and the probability of success is greater than
$0.99$. However, at the 101-th step we should reject an object even if it is
better than all the preceding ones because accepting it would be equivalent to
accepting it if the number of objects was equal to 1000 and, as we know, that
would not be optimal.

In this section we will prove that if the random number of
objects $X$ satisfies certain properties, then the optimal strategy is still a
threshold strategy. Moreover, we will see that the required properties are fulfilled in the case of a Uniform
distribution $\mathcal{U}[1,n]$ as well as in the case of a Poisson
distribution $\mathcal{P}(\lambda)$. We will address the Best-or-Worst and the Postdoc problems separately but all the hard work will be done in the former case.

\subsection{The Best-or-Worst problem}

Recall that in the Best-or-Worst problem, a \emph{nice candidate} is an
object which is either better or worse than all the preceding ones.

\begin{defi}\label{defi1}
If the random number of objects is a discrete random variable $X$, let us
define the following probabilities.
\begin{itemize}
\item $P_{A}^{X}(r)$ is the probability of success if we accept a nice candidate
at the $r$-th step. Note that if $X=k$ and we accept a nice candidate at the
$r$-th step, the probability of success is $r/k$. Thus,
\[
\displaystyle P_{A}^{X}(r)=\mathbb{E}\left( \frac{r}{X} \Big| X\geq r\right) =\displaystyle{\frac{\sum_{k=r}^\infty \frac{r}{k} p(X=k)}{\sum_{k=r}^\infty  p(X=k)}}.
\]
\item $P_{R}^{X}(r)$ is the probability of success if we reject an object
(regardless it is nice or not) at the $r$-th step in order to accept the next
nice candidate to be found. If $X=k$ and we reject an object at the $r$-th
step in order to accept the next nice candidate to be found, the probability of
success is $\displaystyle{\frac{2r(k-r)}{k(k-1)}}$ (see Proposition  \ref{umbral}). Thus,
\[
\displaystyle P_{R}^{X}(r)=\mathbb{E}\left( \frac{2r(X-r)}{X(X-1)} \Big| X\geq
r\right) =\displaystyle{\frac{\sum_{k=r}^\infty  \frac{2r(k-r)}{k(k-1)} \ p(X=k)}{\sum_{k=r}^\infty p(X=k)}}.
\]
\item $\widetilde{P}_{R}^{X}(r)$ is the probability of success if we reject an
object at the $r$-th step in order to adopt the optimal strategy later on. If
we consider $q_{r}=p(X>r|X\geq r)$ it is easy to see that
\[
\widetilde{P}_{R}^{X}(r )= \frac{q_{r}}{r+1} \max\left\{ P_{A}^{X}(r+1
),\widetilde{P}_{R}^{X}(r+1)\right\} +\frac{r q_{r}}{r+1}\widetilde{P}_{R}%
^{X}(r+1).
\]
\end{itemize}
\end{defi}

In this setting and in terms of dynamic programming, the following strategy is
obviously optimal at the $r$-th step:

\begin{itemize}
\item If the $r$-th object is not a nice candidate, then reject it.
\item If the $r$-th object is a nice candidate but $P_{A}^{X}(r)<\widetilde
{P}_{R}^{X}(r)$, then reject it.
\item If the $r$-th object is a nice candidate and $P_{A}^{X}(r)\geq\widetilde
{P}_{R}^{X}(r)$, then accept it.
\end{itemize}

From the very definition it is clear that $P_{R}^{X}(r)\leq\widetilde{P}_{R}^{X}(r)$ for every $r$. It is also clear that, if the range of
$X$ is infinite, then the probability of success rejecting an object at the $r$-th step is strictly positive for every $r$; i.e, $\widetilde{P}_{R}^{X}(r)>0$.
Thus, give any $r$ there must exist $\widehat{r}>r$ such that $P_{A}^{X}(\widehat{r})\geq\widetilde{P}_{R}^{X}(\widehat{r})$ for, otherwise, the optimal strategy would reject every object from the $r$-th step on and we would have that $\widetilde{P}_{R}^{X}(r)=0$ which is obviously a contradiction.

Now, the following result will allow us to work with $P_{R}^{X}(r)$ rather than
with the more complex $\widetilde{P}_{R}^{X}(r)$.

\begin{lem}
\label{lemlim}
Let $X$ be a non-negative discrete random variable such that, either its range is finite or $\displaystyle \liminf P_{A}^{X}(r)>1/2$.
Assume that there exists $r_0$ such that $P_{A}^{X}(r)\geq P_{R}^{X}(r)$ for every $r>r_0$. Then, $P_{A}^{X}(r)\geq\widetilde{P}_{R}^{X}(r)$ for every $r>r_0$.
\end{lem}
\begin{proof}
Given $r_0$, let us consider the set $S=\{r>r_{0}:P_{A}^{X}(r)<\widetilde{P}_{R}^{X}(r)\}$.  We claim that $S$ is bounded. If the range of $X$ is finite, this is trivially the case. If, on the other hand, the range of $X$ is infinite, $\displaystyle \liminf P_{A}^{X}(r)>1/2$ implies that there exists $\widetilde{r}$ such that $P_{A}^{X}(r)>1/2$ for every $r>\widetilde{r}$. Since $P_{A}^{X}(r)+\widetilde{P}_{R}^{X}(r)\leq 1$, this implies that $P_{A}^{X}(r)>\widetilde{P}_{R}^{X}(r)$ for every $r>\widetilde{r}$ and hence $S$ is bounded (by $\widetilde{r}$).

If $S=\emptyset$ the result follows so let us assume that $S$ is nonempty and
let $r^{\prime}$ be its maximum. This means that $\widetilde{P}_{R}%
^{X}(r^{\prime})>P_{R}^{X}(r^{\prime})$ while $P_{A}^{X}(r^{\prime}%
+1)\geq\widetilde{P}_{R}^{X}(r^{\prime}+1)$ but this is a contradiction.

This is because, if the probability of success rejecting an object at the
$r^{\prime}$-th step is bigger than the probability of success rejecting it in
order to accept the next nice candidate; i.e. if $\widetilde{P}_{R}^{X}%
(r^{\prime})>P_{R}^{X}(r^{\prime}$), then accepting a nice candidate at the
next step cannot be optimal; i.e., it is not possible that $P_{A}%
^{X}(r^{\prime}+1)\geq\widetilde{P}_{R}^{X}(r^{\prime}+1)$.
\end{proof}

Now, we can prove the following general result which shows that, under certain
conditions, the optimal strategy is a threshold strategy.

\begin{teor}
\label{umgen} In the Best-or-Worst problem, let
the number of objects $X$ be a non-negative discrete random variable such that, either its range is finite or $\displaystyle \liminf P_{A}^{X}(r)>1/2$. Furthermore, assume that
\[
P_{A}^{X}(r)\geq P_{R}^{X}(r)\Rightarrow P_{A}^{X}(r+1)\geq P_{R}^{X}(r+1).
\]
Then, there exists $r_{0}$ such that the following strategy is optimal:
\begin{enumerate}
\item Reject the $r_{0}$ first inspected objects.
\item After that, accept the first nice candidate which is inspected.
\end{enumerate}
\end{teor}

\begin{proof}
Just consider $r_{0}=\max\{r:P_{A}^{X}(r)< P_{R}^{X}(r)\}$ and apply the
previous lemma.
\end{proof}

The remaining of the section will be devoted to see that we can apply Theorem
\ref{umgen} either if the random number of objects follows a Uniform
distribution $X\sim\mathcal{U}[1,n]$ or a Poisson distribution $X\sim\mathcal{P}%
(\lambda)$. In particular, we will see that in both situations the conditions
from Theorem \ref{umgen} holds.

The following lemma is devoted to explicitly compute $P_{A}^{\mathcal{U}%
[1,n]}(r)$ and $P_{R}^{\mathcal{U}[1,n]}(r)$, which are defined as in Definition \ref{defi1} but for the particular case of $X\sim\mathcal{U}[1,n]$.

\begin{lem}
\label{lemU} Let $\psi$ denote the digamma function. Then,

\begin{itemize}
\item[i)] $\displaystyle P_{A}^{\mathcal{U}[1,n]}(r)=\frac{r\left(
\psi(n+1)-\psi(r)\right)  }{n+1-r}$,

\item[ii)] $\displaystyle P_{R}^{\mathcal{U}[1,n]}(r)=\frac{2r\left(
r-n+n\psi(n)-n\psi(r)\right)  }{n(n+1-r)}$.
\end{itemize}
\end{lem}

\begin{proof}
Let $X\sim\mathcal{U}[1,n]$ be the random variable defining the number of
objects. If $X=k$ (i.e., if there are $k$ objects) and we accept a nice candidate
at the $r$-th step, the probability of success is $r/k$. Thus, taking into
account that
\[
p(X=k | X\geq r)=1/(n+1-r),
\]
we have that
\[
P_{A}^{\mathcal{U}[1,n]}(r)=\sum_{k=r}^{n}\frac{r}{k(n+1-r)}=\frac{r\left(
\psi(n+1)-\psi(r)\right)  }{n+1-r}%
\]
because, for any positive integer $m$ it holds that $\displaystyle\psi
(m)=\sum_{k=1}^{m-1}\frac{1}{k}-\gamma$, ($\gamma$ being the Euler-Mascheroni constant).

On the other hand, if $X=k$ and we reject a nice candidate at the $r$-th step in
order to accept the next nice candidate to be found, the probability of success
is $P_{k}(r)=\displaystyle{\frac{2r(k-r)}{k(k-1)}}$. Hence,
\[
P_{R}^{\mathcal{U}[1,n]}(r):=\sum_{k=r+1}^{n}P_{k}(r)\frac{1}{n+1-r}%
=\frac{2r\left(  r-n+n\psi(n)-n\psi(r)\right)  }{n(n+1-r)}%
\]
using again the definition of the digamma function.
\end{proof}

Once we have computed the values of $P_{A}^{\mathcal{U}%
[1,n]}(r)$ and $P_{R}^{\mathcal{U}[1,n]}(r)$ we can prove the following result that
guarantees that we can apply Theorem \ref{umgen} in the Uniform case.

\begin{prop}
\label{propU} Let $n\in\mathbb{N}$ and $r\in[1,n)$. Then,
\[
P_{A}^{\mathcal{U}[1,n]}(r)>P_{R}^{\mathcal{U}[1,n]}(r)\Rightarrow
P_{A}^{\mathcal{U}[1,n]}(r+1)>P_{R}^{\mathcal{U}[1,n]}(r+1).
\]
\end{prop}
\begin{proof}
It is easy to see that, for every $n/2<r<n$ it holds that $P_{A}%
^{\mathcal{U}[1,n]}(r)>P_{R}^{\mathcal{U}[1,n]}(r)$. Hence, we can restrict
ourselves to the case $1<r\leq n/2$.

Let us assume that $P_{A}^{\mathcal{U}[1,n]}(r)>P_{R}^{\mathcal{U}[1,n]}(r)$
and that $P_{A}^{\mathcal{U}[1,n]}(r+1)\leq P_{R}^{\mathcal{U}[1,n]}(r+1)$.
Using Lemma \ref{lemU}, we have that
\[
P_{R}^{\mathcal{U}[1,n]}(r+1)-P_{A}^{\mathcal{U}[1,n]}(r+1)=\frac
{(r+1)\big((r-n)(2r+1)+nr(\psi(n)-\psi(r))\big)}{n\left(  n-r\right)  r}\geq0
\]
Consequently, we have that
\[
A:=\frac{(r-n)(2r+1)}{nr}+\psi(n)-\psi(r)\geq0.
\]

On the other hand, using Lemma \ref{lemU} again we have that
\[
P_{A}^{\mathcal{U}[1,n]}(r)-P_{R}^{\mathcal{U}[1,n]}(r)=\frac
{r\big(1+2n-2r-n\psi(n)+n\psi(r)\big)}{n(n+1-r)}>0
\]
and, consequently, that
\[
B:=\frac{1+2n-2r}{n}-\psi(n)+\psi(r)>0.
\]

Now, since $A,B>0$, it follows that
\[
0<A+B=\frac{1+2n-2r}{n}+\frac{(r-n)(2r+1)}{nr}=\frac{2r-n}{nr}.
\]
Since this implies that $r>n/2$ we have reached a contradiction and the result follows.
\end{proof}

Now, we turn to the Poisson case. The following lemma is devoted to explicitly compute
$P_{A}^{\mathcal{P}(\lambda)}(r)$ and $P_{R}^{\mathcal{P}(\lambda)}(r)$, which are defined as in Definition \ref{defi1} but for the particular case of $X\sim\mathcal{P}(\lambda)$.

\begin{lem}
\label{lemP} For any $\lambda>0$ let us define
\[
\displaystyle \Psi(r,\lambda):=\sum_{k=r}^{\infty}\frac{{\lambda}^{k}%
}{e^{\lambda}k!}.
\]
Then,
\begin{itemize}
\item[i)] $\displaystyle P_{A}^{\mathcal{P}(\lambda)}(r)=\frac{1}%
{\Psi(r,\lambda)}\sum_{k=r}^{\infty}\frac{r}{k}\frac{{\lambda}^{k}}%
{e^{\lambda}k!}$,
\item[ii)] $\displaystyle P_{R}^{\mathcal{P}(\lambda)}(r)=\frac{1}%
{\Psi(r,\lambda)}\sum_{k=r}^{\infty}\frac{2(k-r)r}{(k-1)k}\frac{{\lambda}^{k}%
}{e^{\lambda}k!}$.
\end{itemize}
\end{lem}
\begin{proof}
Let $X\sim\mathcal{P}(\lambda)$ be the random variable defining the number of
objects. Then,
\[
p(X=k | X\geq r)=\frac{p(X=k)}{p(X\geq r)}=\frac{\frac{{\lambda}^{k}}
{e^{\lambda}k!}}{\Psi(r,\lambda)}%
\]
and the result follows.
\end{proof}

Now that we have explicit expressions for $P_{A}^{\mathcal{P}(\lambda)}(r)$ and $P_{R}^{\mathcal{P}(\lambda)}(r)$, the following results show that the conditions of Theorem
\ref{umgen} also hold in the Poisson case under consideration.

\begin{prop}
\label{limP} For any $\lambda>0$ it holds that
$$\lim_{r\rightarrow\infty} \frac{1}{\Psi(r,\lambda)}\sum_{k=r}^{\infty}\frac{r}{k}\frac{{\lambda}^{k}}{e^{\lambda}k!}=1.$$
\end{prop}
\begin{proof}
Let us denote $\displaystyle S(r,\lambda)=\sum_{k=r}^{\infty}\frac{r}{k}\frac{{\lambda}^{k}}{e^{\lambda}k!}$. Then it is enough to apply the Stolz-Cesàro theorem taking into account that:
\begin{align*}
\Psi(r+1,\lambda)-\Psi(r,\lambda)&=-\frac{\lambda^r}{e^{\lambda}r!},\\
S(r+1,\lambda)-S(r,\lambda)&=\sum_{k=r+1}^{\infty}\frac{\lambda^k}{ke^{\lambda}k!}-\frac{\lambda^r}{e^{\lambda}r!},\\
\end{align*}
\end{proof}
\begin{prop}
\label{propP} Let $\lambda>0$. Then
\[
P_{A}^{\mathcal{P}(\lambda)}(r)>P_{R}^{\mathcal{P}(\lambda)}(r)\Rightarrow
P_{A}^{\mathcal{P}(\lambda)}(r+1)>P_{R}^{\mathcal{P}(\lambda)}(r+1).
\]
\end{prop}

\begin{proof}
As in Proposition \ref{limP}, using the Stolz-Cesàro theorem, it can be easily proved that
$\displaystyle\lim_{r\rightarrow\infty}P_{A}^{\mathcal{P}(\lambda)}(r)=1$ and
that $\displaystyle\lim_{r\rightarrow\infty}P_{R}^{\mathcal{P}(\lambda)}%
(r)=0$. In this situation, the statement is equivalent to prove that there
exists at most one integer $r_{0}\geq1$ such that $P_{A}^{\mathcal{P}%
(\lambda)}(r_{0}-1)\leq P_{R}^{\mathcal{P}(\lambda)}(r_{0}-1)$ and
$P_{A}^{\mathcal{P}(\lambda)}(r_{0})>P_{R}^{\mathcal{P}(\lambda)}(r_{0})$. In
other words, either $P_{A}^{\mathcal{P}(\lambda)}(r)$ is always greater than
$P_{R}^{\mathcal{P}(\lambda)}(r)$ or they \textquotedblleft
intersect\textquotedblright\ just once.

If, for $r\geq1$, we define
\[
f(r):=\sum_{k=r}^{2r-1}\frac{2r-k-1}{k(k-1)}\frac{\lambda^{k}}{k!},\quad
g(r):=\sum_{k=2r}^{\infty}\frac{k-2r+1}{k(k-1)}\frac{\lambda^{k}}{k!}%
\]
it is straightforward to see that $\displaystyle P_{A}^{\mathcal{P}(\lambda
)}(r)-P_{R}^{\mathcal{P}(\lambda)}(r)=\frac{re^{-\lambda}}{\Psi(r,\lambda
)}\left(  f(r)-g(r)\right)  $. Thus, we will see that $f(r)$ and $g(r)$
\textquotedblleft intersect\textquotedblright\ at most once.

Note that, since $P_{A}^{\mathcal{P}(\lambda)}(r)$ is ultimately bigger than
$P_{R}^{\mathcal{P}(\lambda)}(r)$, then $f(r)$ is ultimately bigger than
$g(r)$.

Now,
\begin{align*}
G_{1}(r)  &  :=g(r+1)-g(r)=\sum_{k=2r+1}^{\infty}\frac{-2}{k(k-1)}%
\frac{\lambda^{k}}{k!}-\frac{1}{2r(2r-1)}\frac{\lambda^{2r}}{(2r)!}<0,\\
G_{2}(r)  &  :=G_{1}(r+1)-G_{1}(r)=\\&=\frac{1}{2r(2r-1)}\frac{\lambda^{2r}%
}{(2r)!}+\frac{2}{2r(2r+1)}\frac{\lambda^{2r+1}}{(2r+1)!}+\frac{1}%
{(2r+1)(2r+2)}\frac{\lambda^{2r+2}}{(2r+2)!}>0.
\end{align*}
This means that $g(r)$ strictly decreases (to $0$) and that it is
\textquotedblleft convex\textquotedblright.

On the other hand,
\[
F_{1}(r):=f(r+1)-f(r)=\sum_{k=r+1}^{2r-1}\frac{2}{k(k-1)}\frac{\lambda^{k}%
}{k!}+\frac{1}{2r(2r-1)}\frac{\lambda^{2r}}{(2r)!}-\frac{1}{r}\frac
{\lambda^{r}}{r!}.
\]
This implies that $F_{1}$ changes sign at most once. Since $F_{1}(r)$ is
clearly negative for big values of $r$, this means that $f(r)$ is either
strictly decreasing or it first increases and then decreases with only one
change in its monotony. Furthermore,
\[
F_{2}(r):=F_{1}(r+1)-F_{1}(r)=G_{2}(r)+\frac{\lambda^{r}}{r(r+1)^{2}r!}%
\phi(r),
\]
with $\displaystyle\phi(r)=\big(r^{2}+(2-\lambda)r+(1-2\lambda)\big)$. This
implies that $F_{2}(r)>G_{2}(r)$ for every $r>\tilde{r}$, with $\tilde{r}$ the
biggest root of $\phi(r)=0$. In other words, for $r>\tilde{r}$, $f(r)$ is
\textquotedblleft more convex\textquotedblright\ than $g(r)$.

Finally, assume that $f(r)$ and $g(r)$ intersect at some point in which both
of them are decreasing. Since $f(r)$ must be ultimately bigger that $g(r)$,
this contradicts either the fact that $f(r)$ only has at most one change in
monotony or the fact that $F_{2}(r)>G_{2}(r)$ for some moment on. This means
that either $f(r)$ and $g(r)$ do not intersect at all or that they do so only
once, which is what we wanted to prove.
\end{proof}

\begin{rem}
We have just seen that $P_{A}^{\mathcal{P}(\lambda)}(r)$ and $P_{R}%
^{\mathcal{P}(\lambda)}(r)$ intersect at most once. Since $P_{A}%
^{\mathcal{P}(\lambda)}(r)$ increases monotonically to 1 while
$P_{R}^{\mathcal{P}(\lambda)}(r)$ tends to 0, the intersection of both
functions depends only on the relationship between $P_{A}^{\mathcal{P}%
(\lambda)}(1)$ and $P_{R}^{\mathcal{P}(\lambda)}(1)$. Namely, $P_{A}%
^{\mathcal{P}(\lambda)}(r)>P_{R}^{\mathcal{P}(\lambda)}(r)$ for every $r$ if
and only if $P_{A}^{\mathcal{P}(\lambda)}(1)>P_{R}^{\mathcal{P}(\lambda)}(1)$. Solving the equation $P_{A}^{\mathcal{P}(\lambda)}(1)=P_{R}^{\mathcal{P}%
(\lambda)}(1)$ leads to the approximate value of $\lambda_{0}=2.2197719\dots$
which means that, for every $\lambda<\lambda_{0}$ it holds that the functions
$P_{A}^{\mathcal{P}(\lambda)}(r)$ and $P_{R}^{\mathcal{P}(\lambda)}(r)$ never
intersect and for every $\lambda\geq\lambda_{0}$ they intersect exactly once.
\end{rem}

Finally, we can give the main result of this section to establish that the
optimal strategy in both considered situations is a threshold strategy.

\begin{teor}
\label{thBW} In the Best-or-Worst problem
let the number of objects follow a Uniform distribution $\mathcal{U}[1,n]$
(resp. a Poisson distribution $\mathcal{P}(\lambda)$). Then, there exists
$r(n)$ (resp. $r(\lambda)$) such that the following strategy is optimal:
\begin{enumerate}
\item Reject the $r(n)$ (resp. $r(\lambda)$) first inspected objects.
\item After that, accept the first nice candidate which is inspected.
\end{enumerate}
We will refer to the value $r(n)$ (resp. $r(\lambda)$) as the optimal cutoff value.
\end{teor}
\begin{proof}
It follows from Theorem \ref{umgen}. If $X\sim\mathcal{U}[1,n]$ it is enough to apply Proposition
\ref{propU} because the range of $X$ is finite. On the other hand, if $X\sim\mathcal{P}(\lambda)$, we have to apply both Proposition \ref{limP} and Proposition \ref{propP}.
\end{proof}

\begin{rem}
The proof of Theorem \ref{thBW} can also be approached in terms of Markov
chains as it was done in \cite{sonin} for the classical Secretary problem. In
fact, in order to prove that the optimal strategy is a threshold strategy, the
only relevant factor is the function that determines the probability of
success if we accept a nice candidate at the $r$-th step with a known number of
objects $k$. As it turns out, this function is the same in the classical
problem as in the Best-or-Worst problem. Thus, the proof would go just as in
the aforementioned paper \cite{sonin}. However, we decide to provide full
explicit proofs, avoiding Markov chains, to keep the paper self-contained and
elementary in nature.
\end{rem}

\subsection{The Postdoc problem}

Now we turn to the Postdoc problem. In this setting, a \emph{nice candidate} is an
object which is the second better than all the preceding ones. First of all, we have the following analogue to Definition \ref{defi1}.

\begin{defi}\label{defi2}
If the random number of objects is a discrete random variable $X$, let us
define the following probabilities.
\begin{itemize}
\item $\mathbf{P}_{A}^{X}(r)$ is the probability of success if we accept a nice candidate
at the $r$-th step. Note that if $X=k$ and we accept a nice candidate at the
$r$-th step, the probability of success is $r/k$. Thus,
\[
\displaystyle \mathbf{P}_{A}^{X}(r)=\mathbb{E}\left(\frac{r(r-1)}{X(X-1)} \Big| X\geq r\right)=\displaystyle{\frac{\sum_{k=r}^\infty \frac{r(r-1)}{k(k-1)} p(X=k)}{\sum_{k=r}^\infty  p(X=k)}}.
\]
\item $\mathbf{P}_{R}^{X}(r)$ is the probability of success if we reject an object
(regardless it is nice or not) at the $r$-th step in order to accept the next
nice candidate to be found. If $X=k$ and we reject an object at the $r$-th
step in order to accept the next nice candidate to be found, the probability of
success is $\displaystyle{\frac{r(k-r)}{k(k-1)}}$ (see Proposition  \ref{umbral}). Thus,
\[
\displaystyle \mathbf{P}_{R}^{X}(r)=\mathbb{E}\left(\frac{r(X-r)}{X(X-1)} \Big| X\geq r\right)=\displaystyle{\frac{\sum_{k=r}^\infty  \frac{r(k-r)}{k(k-1)} \ p(X=k)}{\sum_{k=r}^\infty p(X=k)}}.
\]
\item $\widetilde{\mathbf{P}}_{R}^{X}(r)$ is the probability of success if we reject an
object at the $r$-th step in order to adopt the optimal strategy later on. If
we consider $q_{r}=p(X>r|X\geq r)$ it is easy to see that
\[
\widetilde{\mathbf{P}}_{R}^{X}(r)= \frac{q_{r}}{r+1} \max\left\{ \mathbf{P}_{A}^{X}(r+1
),\widetilde{\mathbf{P}}_{R}^{X}(r+1)\right\} +\frac{r q_{r}}{r+1}\widetilde{\mathbf{P}}_{R}%
^{X}(r+1).
\]
\end{itemize}
\end{defi}

After these definitions, we could state and prove the direct analogues to Lemma \ref{lemlim} and to Theorem \ref{umgen}. On the other hand, note that it is straightforward to check that
$$\mathbf{P}_{A}^{X}(r)-\mathbf{P}_{R}^{X}(r)=P_{A}^{X}(r)-P_{R}^{X}(r).$$
Consequently, the corresponding analogues to Propositions \ref{propU} and \ref{propP} also hold in this setting. Finally, as in Proposition \ref{limP}, it can be easily proved using the Stolz-Cesàro theorem that
$$\lim_{r\to\infty}\mathbf{P}_{A}^{\mathcal{P}(\lambda)}(r)=1.$$

This being said, it follows that the following analogue to Theorem \ref{thBW} also holds, showing that in the Postdoc problem the optimal strategy is still a threshold strategy.

\begin{teor}
\label{thPD}
In the Postdoc problem
let the number of objects follow a Uniform distribution $\mathcal{U}[1,n]$
(resp. a Poisson distribution $\mathcal{P}(\lambda)$). Then, there exists
$r(n)$ (resp. $r(\lambda)$) such that the following strategy is optimal:
\begin{enumerate}
\item Reject the $r(n)$ (resp. $r(\lambda)$) first inspected objects.
\item After that, accept the first nice candidate which is inspected.
\end{enumerate}
We will refer to the value $r(n)$ (resp. $r(\lambda)$) as the optimal cutoff value.
\end{teor}

Now that we have seen than both in the Best-or-Worst problem and in the Postdoc problem the optimal strategies are threshold strategies, we are in the conditions to apply Corollary \ref{corred}. Thus, we can focus just on one of the problems (we choose the Best-or-Worst problem). The forthcoming sections will be devoted to study the optimal cutoff values as well as the associated probabilities of success for each of the considered distributions.

\section{The Best-or-Worst problem when the number of objects follows a
Uniform distribution $\mathcal{U}[1,n]$}

Taking into account Proposition \ref{umbral}, if there were $k>1$ objects, the probability of
success using a threshold strategy with cutoff value $1\leq r<k$ would be
$P_{k}(r)=\frac{2r(k-r)}{k(k-1)}$. On the other hand, if the random variable
defining number of objects follows a discrete Uniform distribution
$X\sim\mathcal{U}[1,n]$, we have that $p(X=k)=1/n$ for every $k$. Hence, the
probability of success is given in this situation by the function
\[
F^{\mathcal{U}}(r,n)=\sum_{k=r+1}^{n} P_{k}(r)p(X=k)=\sum_{k=r+1}^{n}%
\frac{P_{k}(r)}{n}= \sum_{k=r+1}^{n} \frac{2r(k-r)}{k(k-1)n}.
\]
Now, let us denote by $\mathcal{M}(n)\in[1,n]$ the optimal cutoff value; i.e.,
the value for which the function $F^{\mathcal{U}}(\cdot,n)$ reaches its
maximum. Also, let us denote by $P(n):=F^{\mathcal{U}}(\mathcal{M}(n),n)$.
Thus, $P(n)$ denotes the probability of success in the Best-or-Worst problem
when the number of objects follows a discrete Uniform distribution
$\mathcal{U}[1,n]$ using the optimal threshold strategy.

\begin{rem}
With the previous notation, it is straightforward to see that $\mathcal{M}%
(1)=\mathcal{M}(2)=0$ and also that $P(1)=P(2)=1$. This corresponds to the
fact that, in the Best-or-Worst problem, if there is only one or two objects,
we will always succeed if we accept the first one.
\end{rem}

In order to study the behavior of $\mathcal{M}(n)$ and $P(n)$ we shall first
prove that, with the only exception of the previous remark, $P$ is strictly
decreasing. To do so we first prove an adaptation of the strategy-stealing
argument in the following lemma.

\begin{lem}
\label{L:in} For every pair of integers $(r,n)$ with $1<r<n$, one of the
following identities holds:
\begin{align*}
F^{\mathcal{U}}(r,n+1)  &  <F^{\mathcal{U}}(r-1,n),\\
F^{\mathcal{U}}(r,n+1)  &  <F^{\mathcal{U}}(r,n).
\end{align*}

\end{lem}

\begin{proof}
Let us denote $H(n,r)=\displaystyle\sum_{k=r}^{n-1}\frac{1}{i}$. Then,
\begin{align*}
F^{\mathcal{U}}(r,n)  &  =\sum_{k=r+1}^{n}\frac{2k(k-r)}{k(k-1)n}=\frac{2r}%
{n}\sum_{k=r+1}^{n}\left(  \frac{r}{k}+\frac{1}{k-1}+\frac{r}{k-1}\right)  =\\
&  =\frac{2r}{n}\left(  \left(  \frac{r}{n}-1\right)  +\sum_{k=r+1}^{n}%
\frac{1}{k-1}\right)  =\frac{2r}{n}\left(  \left(  \frac{r}{n}-1\right)
+H(n,r)\right)  .
\end{align*}
In the same way we have that
\begin{align*}
F^{\mathcal{U}}(r,n+1)  &  =\frac{2r}{n+1}\left(  \left(  \frac{r}%
{n+1}-1+\frac{1}{n}\right)  +H(n,r)\right)  ,\\
F^{\mathcal{U}}(r-1,n)  &  =\frac{2(r-1)}{n}\left(  \left(  \frac{r-1}%
{n}-1+\frac{1}{r-1}\right)  +H(n,r)\right)  .
\end{align*}
And, as a consequence, it follows that
\begin{align*}
F^{\mathcal{U}}(r,n+1)-F^{\mathcal{U}}(r-1,n)  &  =\frac{2(1+n-r)}%
{n(n+1)}\left(  -\frac{(1+2n)(1+n-r)}{n(n+1)}+H(n,r)\right)  ,\\
F^{\mathcal{U}}(r,n+1)-F^{\mathcal{U}}(r,n)  &  =\frac{2r}{n(n+1)}\left(
\frac{2n(1+n-r)-r}{n(n+1)}-H(n,r)\right)
\end{align*}

Now, if we assume that both $\displaystyle F^{\mathcal{U}}%
(r,n+1)-F^{\mathcal{U}}(r-1,n)\geq0$ and $\displaystyle F^{\mathcal{U}%
}(r,n+1)-F^{\mathcal{U}}(r,n)\geq0 $, we get that
\begin{align*}
A  &  :=H(n,r)n(n+1)-(1+2n)(1+n-r)\geq0,\\
B  &  :=-\left(  H(n,r)n\,\left(  n+1\right)  \right)  +2n\left(
1+n-r\right)  -r\geq0.
\end{align*}
As a consequence, it follows that $0\leq A+B=-1-n<0$. This is a contradiction
and hence the result.
\end{proof}

Using this lemma, we can now prove that $P$ is decreasing.

\begin{prop}
\label{C:dec} Let $n>1$ be any integer. Then, $P(n+1)<P(n)$.
\end{prop}

\begin{proof}
If $n>1$,
\begin{align*}
P(n+1)  &  =F^{\mathcal{U}}(\mathcal{M}(n+1),n+1)<\max\{F^{\mathcal{U}%
}(\mathcal{M}(n+1),n),F^{\mathcal{U}}(\mathcal{M}(n+1)-1,n)\}\leq\\
&  \leq F^{\mathcal{U}}(\mathcal{M}(n),n)=P(n),
\end{align*}
where the first inequality follows from Lemma \ref{L:in} and the second holds
by the definition of $\mathcal{M}(n)$.
\end{proof}

In order to provide further information about $\mathcal{M}(n)$ and $P(n)$ we
first need two technical results. The first one was proved in
\cite[Proposition 1]{bayon}, while the second is just an elementary Calculus exercise. Recall that for every $x>-1/e$, the principal branch of the Lambert $W$-function is the only real number $W(x)>-1$ such that $x=W(x)e^{W(x)}$.

\begin{lem}
\label{conv} Let $\{F_{n}\}$ be a sequence of real functions with $F_{n}%
\in\mathcal{C}[0,n]$ and let $\mathcal{M}(n)$ be the value for which the
function $F_{n}$ reaches its maximum. Assume that the sequence of functions
$\{g_{n}\}_{n\in\mathbb{N}}$ given by $g_{n}(x):=F_{n}(nx)$ converges
uniformly on $[0,1]$ to a function $g$ and that $\theta$ is the only global
maximum of $g$ in $[0,1]$. Then,

\begin{itemize}
\item[i)] $\displaystyle\lim_{n} \mathcal{M}(n)/n =\theta$.

\item[ii)] $\displaystyle\lim_{n} F_{n}(\mathcal{M}(n))= g(\theta)$.

\item[iii)] If $\mathfrak{M}(n)\sim\mathcal{M}(n)$ then $\displaystyle\lim
_{n}F_{n}(\mathfrak{M}(n))=g(\theta)$.
\end{itemize}
\end{lem}

\begin{lem}
\label{L:max} The function $g(x)=-2x\log x-2x(1-x)$ reaches its absolute
maximum in the interval $[0,1]$ at the point
\[
\vartheta:=-\frac{1}{2}W(-\frac{2}{e^{2}})=0.20318786\dots,
\]
where $W$ denotes Lambert $W$-function. Moreover, the value of this maximum
is:
\[
g(\vartheta)=2(\vartheta-\vartheta^{2})=0.32380511\dots
\]

\end{lem}

Now, the following result provides estimations for the values $\mathcal{M}(n)$
and describes the asymptotic behavior of $P$.

\begin{teor}
\label{TBWU} With all the previous notation, the following hold:

\begin{itemize}
\item[i)] For every positive integer $n$,
\[
1=P(1)= P(2)>P(3)>\cdots>P(n)>P(n+1)>\cdots>2(\vartheta-\vartheta^{2}).
\]

\item[ii)] $\displaystyle \lim_{n} \frac{\mathcal{M}(n)}{n}=\vartheta$; i.e.,
$\mathcal{M}(n)\sim\vartheta n.$

\item[iii)] $\displaystyle \lim_{n} F^{\mathcal{U}}(\lfloor\vartheta\cdot
n\rfloor,n)=\lim_{n\rightarrow\infty}P(n)=2(\vartheta-\vartheta^{2})$.
\end{itemize}
\end{teor}

\begin{proof}
As in the proof of Lemma \ref{L:in}, we have that
\begin{align*}
F^{\mathcal{U}}(r,n)  &  =\frac{2r}{n}\left(  \left(  \frac{r}{n}-1\right)
+\sum_{k=r}^{n-1}\frac{1}{k}\right)
\end{align*}
so, if we recall the definition of the digamma function $\displaystyle\psi(n):=\sum_{k=1}^{n-1}\frac{1}{k}-\gamma$, we
get that
\begin{align*}
F^{\mathcal{U}}(r,n)& =\frac{2r\left(  r-n+n\psi(n)-n\psi(r)\right)  }{n^{2}}%
\end{align*}
Now, if we define $g_{n}(x):=F^{\mathcal{U}}(nx,n)$ we have that the sequence
of functions $\{g_{n}\}$ converges uniformly to the function $g(x):=2x\left(
-1+x-\log(x)\right)  $.

Since $P(n)=g_{n}\left(  \frac{\mathcal{M}(n)}{n}\right)  $ it is enough to
apply Proposition \ref{C:dec} and Lemma \ref{L:max} to obtain point i). Points
ii) and iii) readily follow from Lemma \ref{conv} and Lemma \ref{L:max} and
the proof is complete.
\end{proof}

Theorem \ref{TBWU} shows that $[\vartheta\cdot n]$ (the nearest integer to
$\vartheta\cdot n$) constitutes a practical estimation of the optimal cutoff
value in the optimal threshold strategy and that, following this strategy, the
probability of success is greater than $2(\vartheta-\vartheta^{2}%
)=0.3238\dots$. The first few values of $n$ for which the estimation
$[n\vartheta]$ fails are
\[
8,13,18,23,32,37,42,47,52,57,62, 67,72,77,82,96,101,106,111,116,121,\dots
\]
Even if the estimation fails about 20\% of times, $[n\vartheta]$ differs at
most 1 from the actual optimal cutoff value and the error is negligible if
compared with the probability of success of the optimal strategy for large
values of $n$.

We are now interested in finding a better estimate for $\mathcal{M}(n)$. As
usual, we first need to introduce a technical result.

\begin{lem}
\label{lem:f} Let us consider the function $f(r,n):=-\frac{2r}{n}+\frac
{2r^{2}}{n^{2}}+\frac{2r}{n}\psi(n)-\frac{2r\log(r)}{n}+\frac{1}{n}$ and let
$\alpha(n)\in[1,n]$ be the value for which the function $f(\cdot,n)$ reaches
its maximum. Then, $\alpha(n)\approx\mathcal{M}(n)$.
\end{lem}

\begin{proof}
As we saw int he proof of Theorem \ref{TBWU},
\[
F^{\mathcal{U}}(r,n)=-\frac{2r}{n}+\frac{2r^{2}}{n^{2}}+\frac{2r}{n}%
\psi(n)-\frac{2r}{n}\psi(r).
\]
On the other hand, for any integer $r$ it holds that $\psi(r)=\log r-\frac
{1}{2r}+\epsilon(r)$, with $\epsilon(r)=O(1/2r)$ if $r\rightarrow\infty$.
Thus,
\[
F^{\mathcal{U}}(r,n)=-\frac{2r}{n}+\frac{2r^{2}}{n^{2}}+\frac{2r}{n}%
\psi(n)-\frac{2r\log(r)}{n}+\frac{1}{n}-\frac{2r\epsilon(r)}{n}=f(r,n)-\frac
{2r\epsilon(r)}{n}.
\]

Since $r\geq1$, it follows that there exists $k$ such that $\left\vert
2r\epsilon(r)\right\vert \leq k$ for every $r$ and hence:
\[
g_{1}(r,n)=f(r,n)-\frac{k}{n}\leq P(r,n)\leq f(r,n)+\frac{k}{n}=g_{2}(r,n).
\]
Obviously both functions $g_{1}(r,n)$ and $g_{2}(r,n)$ reach their maximum at
$\alpha(n)$. In this situation, there exist points $\beta_{1}(n)$ and
$\beta_{2}(n)$ such that $g_{2}(\beta_{1}(n),n)=g_{2}(\beta_{2}(n),n)=g_{1}%
(\alpha(n),n)$ and $\alpha(n)\in\left(  \beta_{1}(n),\beta_{2}(n)\right)  $
and the inequality above implies that
\[
\beta_{1}(n)\leq\mathcal{M}(n)\leq\beta_{2}(n).
\]

Finally, for every $r$ we have that $g_{2}(r,n)-g_{1}(r,n)=2k/n$ so it follows
that $\left\vert \beta_{2}(n)-\beta_{1}(n))\right\vert \underset
{n\rightarrow\infty}{\longrightarrow}0$ and, consequently also $\left\vert
\alpha(n)-\mathcal{M}(n)\right\vert \underset{n\rightarrow\infty
}{\longrightarrow} 0$ as we wanted to prove.
\end{proof}

As a consequence of the previous result, if we compute the value of
$\alpha(n)$ we can give another estimation for $\mathcal{M}(n)$. In fact, we
have the following result.

\begin{teor}
With all the previous notation, the following hold:

\begin{itemize}
\item[i)] $\displaystyle\mathcal{M}(n)\approx-\frac{n}{2} W\left(
-\frac{2e^{-2+\psi(n)}}{n}\right)  $,

\item[ii)] $\displaystyle \mathcal{M}(n)\approx n \vartheta+\frac
{1}{4-2e^{2-2\vartheta}}$.
\end{itemize}
\end{teor}

\begin{proof}
\begin{itemize}

\item[i)] Since $f(\cdot,n)$ (see Lemma \ref{lem:f}) reaches its maximum at
$\alpha(n)$, it follows that
\[
0=\frac{\partial f}{\partial r}(\alpha(n),n)=-\frac{4}{n}+\frac{4\alpha
(n)}{n^{2}}+\frac{2}{n}\psi(n)-\frac{2\log(\alpha(n))}{n}.
\]
From this, and taking into account the definition of Lambert $W$-function it
follows that
\[
\alpha(n)=-\frac{n}{2}W(-\frac{2e^{-2+\psi(n)}}{n})
\]
so it is enough to apply Lemma \ref{lem:f}.

\item[ii)] Using i), we have that
\[
\lim_{n}\left(  \mathcal{M}(n)-n\vartheta\right)  =\lim_{n}\left(  -\frac
{n}{2}W(-\frac{2e^{-2+\psi(n)}}{n})-n\vartheta\right)  =\frac{1}%
{4-2e^{2-2\vartheta}}.
\]
Thus,
\[
\mathcal{M}(n)\approx n\vartheta+\frac{1}{4-2e^{2-2\vartheta}}%
\]
as claimed.
\end{itemize}
\end{proof}

As far as our computing capabilities let us check, the estimation
$\mathcal{M}(n)\approx\left(  n\vartheta+\frac{1}{4-2e^{2-2\vartheta}}\right)
$ fails for very few values of $n$ (only four cases have been found: 2, 3, 23
and 2971). On the other hand, we have not found any value of $n>4$ for which
the estimation $\mathcal{M}(n)\approx\left[  -\frac{n}{2} W\left(
-\frac{2e^{-2+\psi(n)}}{n}\right)  \right]  $ fails.

\section{The Best-or-Worst problem when the number of objects follows a
Poisson distribution $\mathcal{P}(\lambda)$}

Now, let us assume that the random variable defining number of objects follows
a Poisson distribution $X\sim\mathcal{P}(\lambda)$. Hence, $p(X=k)=\frac
{\lambda^{k}}{e^{\lambda}k!}$ and reasoning like at the beginning of the
previous section, we conclude that the probability of success using a
threshold strategy with cutoff value $1\leq r<k$ is given in this situation by
the function
\[
F^{\mathcal{P}}(r,\lambda)=\sum_{k=r+1}^{n}P_{k}(r)p(X=k)=\sum_{k=r+1}%
^{\infty}\frac{2r(k-r)}{k(k-1)}\frac{\lambda^{k}}{e^{\lambda}k!}.
\]

and%
\[
F^{\mathcal{P}}(0,\lambda)=\frac{\lambda}{e^{\lambda}}+\frac{\lambda^{2}%
}{2e^{\lambda}}+\sum_{k=3}^{\infty}\frac{2}{k}\frac{\lambda^{k}}{e^{\lambda
}k!}%
\]

Also, following the same notation as in the previous section, let us denote by
$\mathcal{M}(\lambda)$ the value for which $F^{\mathcal{P}}(\cdot,\lambda)$
reaches its maximum value and let $P(\lambda)=F^{\mathcal{P}}(\mathcal{M}%
(\lambda),\lambda)$; i.e., $P(\lambda)$ denotes the probability of success in
the Best-or-Worst problem when the number of objects follows a Poisson
distribution $\mathcal{P}(\lambda)$ using the optimal threshold strategy.

\begin{figure}[h]
\includegraphics[height=1.6031 in,width=2.4981 in]{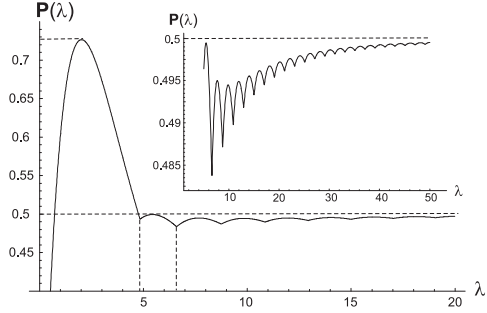}\caption{Graph of
the function $P(\lambda)$}%
\end{figure}

\begin{rem}
The value $\lambda_{m}$ for which the probability of success $P(\lambda)$
reaches its maximum can be explicitly computed solving the equation
\[
-2+2e^{x}-x+2\gamma x+x^{2}+2x\Gamma(-x)+2x\log(-x)=0,
\]
which leads to an approximate value of $\lambda_{m}=2.01771\dots$ and
$P(\lambda_{m})=0.72647\dots$ as can be seen in Figure 1.
\end{rem}

As we can see in Figure 1, the graph of $P(\lambda)$ consists of a sequence of
concave arcs. The rest of the section will be devoted to provide an estimation
for $\mathcal{M}(\lambda)$ and to study the asymptotic behavior of
$P(\lambda)$. In particular, we will see that, as suggested by Figure 1,
$\displaystyle\lim_{\lambda\to\infty}P(\lambda)=1/2$. But first we need some
technical results.

\begin{lem}
\label{LEM:f} Let us consider the function
\[
f(r,\lambda):=\sum_{k=2}^{r}\frac{2r(k-r)}{k(k-1)}\frac{\lambda^{k}%
}{e^{\lambda}k!}.
\]
Then, $\displaystyle\lim_{\lambda\to\infty}f\left(  \lambda/2,\lambda\right)
=0$.
\end{lem}

\begin{proof}
First of all, note that
\[
\left\vert \sum_{k=2}^{\lambda/2}\frac{2(\lambda/2)(k-\lambda/2)}{k(k-1)}%
\frac{\lambda^{k}}{e^{\lambda}k!}\right\vert <\lambda^{2}e^{-\lambda}%
\sum_{k=2}^{\lambda/2}\frac{\lambda^{k}}{k!}<\frac{1}{\lambda}.
\]
Thus, if we define
\[
a_{n}:=2^{2}n^{2}e^{-2n}\sum_{k=2}^{n}\frac{2^{k}n^{k}}{k!},
\]
we just need to prove that $\displaystyle\lim_{n}a_{n}=0$.

Now,
\[
a_{n+1}=2^{2}(n+1)^{2}e^{-2(n+1)}\sum_{k=2}^{n+1}\frac{2^{k}(n+1)^{k}}{k!}%
\]
and since
\begin{align*}
\sum_{k=2}^{n+1}\frac{2^{k}(n+1)^{k}}{k!}  &  =\sum_{k=2}^{n}\frac{2^{k}n^{k}%
}{k!}+\frac{2^{2}}{2!}(2n+1)+\frac{2^{3}}{3!}(3n^{2}+3n+1)+\cdots+\\
&  +\frac{2^{n}}{n!}\left(  n^{n}+\frac{n}{1!}n^{n-1}+\frac{n(n-1)}{2!}%
n^{n-2}+\cdots+\frac{n(n-1)\cdots2.1}{n!}\right)  +\\
&  +\frac{2^{n+1}}{(n+1)!}(n+1)^{n+1},
\end{align*}
we have that
\begin{align*}
a_{n+1}  &  =2^{2}(n+1)^{2}e^{-2(n+1)}\sum_{k=2}^{n}\frac{2^{k}n^{k}}%
{k!}+2^{2}(n+1)^{2}e^{-2(n+1)}\frac{2^{2}}{2!}(2n+1)+\cdots+\\
&  +2^{2}(n+1)^{2}e^{-2(n+1)}\frac{2^{n}}{n!}\left(  n^{n}+\frac{n}{1!}%
n^{n-1}+\frac{n(n-1)}{2!}n^{n-2}+\cdots+\frac{n(n-1)\cdots2.1}{n!}\right)  +\\
&  +2^{2}(n+1)^{2}e^{-2(n+1)}\frac{2^{n+1}}{(n+1)!}(n+1)^{n+1}.
\end{align*}

On the other hand,
\[
\sum_{k=2}^{n}\frac{2^{k}n^{k}}{k!}=\frac{a_{n}}{2^{2}n^{2}e^{-2n}}%
\]
so
\begin{align*}
a_{n+1}  &  =\left(  \frac{n+1}{n}\right)  ^{2}\frac{a_{n}}{e^{2}}%
+2^{2}(n+1)^{2}e^{-2(n+1)}\frac{2^{2}}{2!}(2n+1)+\cdots+\\
&  +2^{2}(n+1)^{2}e^{-2(n+1)}\frac{2^{n}}{n!}\left(  n^{n}+\frac{n}{1!}%
n^{n-1}+\frac{n(n-1)}{2!}n^{n-2}+\cdots+\frac{n(n-1)\cdots2.1}{n!}\right)  +\\
&  +2^{2}(n+1)^{2}e^{-2(n+1)}\frac{2^{n+1}}{(n+1)!}(n+1)^{n+1},
\end{align*}
which clearly implies that
\begin{equation}
\label{eqlim}\lim_{n} \left(  a_{n+1}-\left(  \frac{n+1}{n}\right)  ^{2}%
\frac{a_{n}}{e^{2}}\right)  =0.
\end{equation}

Finally, since $a_{n}$ is decreasing and non-negative, we have that
$\displaystyle\lim_{n}a_{n+1}=\lim_{n}a_{n}=l$. Consequently, Equation
\ref{eqlim} implies that $l=l/e^{2}$; i.e., $l=0$ as claimed.
\end{proof}

\begin{teor}
\label{TEOR:PE} Let us consider the function
\[
f^{\star}(r,\lambda):=\sum_{k=2}^{\infty}\frac{2r(k-r)}{k(k-1)}\frac
{\lambda^{k}}{e^{\lambda}k!}.
\]
For every $\lambda>0$, let us denote by $\mathcal{M}^{\star}(\lambda)$ the
value for which $f^{\star}(\cdot,\lambda)$ reaches its maximum and let
$P^{\star}(\lambda)=f^{\star}(\mathcal{M}^{\star}(\lambda),\lambda)$. Then,
the following hold:

\begin{itemize}
\item[i)] $\displaystyle\lim_{\lambda\to\infty}\mathcal{M}^{\star}%
(\lambda)/\lambda=1/2$.

\item[ii)] $\displaystyle \lim_{\lambda\to\infty}f^{\star}(\lambda
/2,\lambda)=1/2$.

\item[iii)] $\mathcal{M}^{\star}(\lambda)\approx\lambda/2-1$.
\end{itemize}
\end{teor}

\begin{proof}
First of all, we are going to compute the value of $\mathcal{M}^{\star
}(\lambda)$. To do so, let us consider the following functions:
\begin{align*}
S_{1}(x)  &  =\sum_{k=2}^{\infty}\frac{x^{k}}{(k-1)k!},\\
S_{2}(x)  &  =\sum_{k=2}^{\infty}\frac{x^{k}}{k(k-1)k!}=-\sum_{k=2}^{\infty
}\frac{x^{k}}{kk!}+S_{1}(x),\\
S_{3}(x)  &  =\sum_{k=2}^{\infty}\frac{x^{k}}{kk!}.
\end{align*}
Then, we have that
\[
S_{1}^{\prime\prime}(x)=\frac{e^{x}-1}{x},\quad S_{3}^{\prime}(x)=\frac
{e^{x}-1}{x}-1
\]
and, by integrating these expressions we obtain that
\begin{align*}
S_{1}(x)  &  =1-e^{x}+x-\gamma x+x\operatorname{E}(x)-x\log x,\\
S_{3}(x)  &  =-\gamma+\operatorname{E}(x)-\log x-x,
\end{align*}
where
\[
\operatorname{E}(\lambda)=\gamma+\log\lambda+\int_{0}^{\lambda} \frac{e^{x}%
-1}{x}dx.
\]
Also, as a consequence we obtain that
\[
S_{2}(x)=1-e^{x}+2x+\left(  \gamma-\operatorname{E}(x)+\log x\right)  \left(
1-x\right)  .
\]

Now, we observe that
\[
f^{\star}(r,\lambda)=2re^{-\lambda}\left(  \sum_{k=2}^{\infty}\frac
{\lambda^{k}}{(k-1)k!}-r\sum_{k=2}^{\infty}\frac{\lambda^{k}}{k(k-1)k!}%
\right)  =2re^{-\lambda}S_{1}(\lambda)-2r^{2}e^{-\lambda}S_{2}\lambda).
\]
Thus,
\[
\frac{\partial f^{\star}(r,\lambda)}{\partial r}=2e^{-\lambda}S_{1}%
(\lambda)-4re^{-\lambda}S_{2}(\lambda)
\]
so, if we define
\[
r_{\lambda}:=\frac{1-e^{\lambda}+\lambda-\gamma\lambda+\lambda\operatorname{E}%
(\lambda)-\lambda\log\lambda}{2\left(  1-e^{\lambda}+2\lambda+\left(
\gamma-\operatorname{E}(\lambda)+\log\lambda\right)  (1-\lambda)\right)  },
\]
it is clear that $\frac{\partial f^{\star}(r,\lambda)}{\partial r}%
|_{r=r_{\lambda}}=0$ and, consequently, we have just obtained that
$\mathcal{M}^{\star}(\lambda)=r_{\lambda}$.

Once we have computed the value of $\mathcal{M}^{\star}(\lambda)$, we are in
the condition to prove the three statements of the theorem.

\begin{itemize}
\item[i)] If we recall the definition of $\operatorname{E}(\lambda)$, we have
that
\begin{align*}
\lim_{\lambda\rightarrow\infty}\frac{\mathcal{M}^{\star}(\lambda)}{\lambda}
&  =\frac{1}{2}\lim_{\lambda\rightarrow\infty}\frac{1-e^{\lambda}%
+\lambda-\gamma\lambda+\lambda\operatorname{E}(\lambda)-\lambda\log\lambda
}{\lambda\,\left(  1-e^{\lambda}+2\lambda+\left(  \gamma-\operatorname{E}%
(\lambda)+\log\lambda\right)  (1-\lambda)\right)  }=\\
&  =\frac{1}{2}\lim_{\lambda\rightarrow\infty}\frac{1-e^{\lambda}%
+\lambda+\lambda\int_{0}^{\lambda}\frac{e^{x}-1}{x}dx}{\lambda\,\left(
1-e^{\lambda}+2\lambda+(\lambda-1)\int_{0}^{\lambda}\frac{e^{x}-1}%
{x}dx\right)  }%
\end{align*}
so, applying L'Hôpital's rule repeatedly we obtain:
\[
\lim_{\lambda\rightarrow\infty}\frac{\mathcal{M}^{\star}(\lambda)}{\lambda
}=\frac{1}{2}\lim_{\lambda\rightarrow\infty}\frac{1+\lambda}{\lambda+3}%
=\frac{1}{2},
\]
as claimed.

\item[ii)] We have that
\[
\lim_{\lambda\to\infty}f^{\star}(\lambda/2,\lambda)=\lim_{\lambda\to\infty
}\frac{\lambda}{2e^{\lambda}}\left(  (-\lambda^{2}+3\lambda)\int_{0}^{\lambda
}\frac{e^{x}-1}{x}dx+(\lambda-2)e^{\lambda}-2\lambda^{2}+\lambda+1)\right)
\]
so, using L'Hôpital's rule, we get that:
\[
\lim_{\lambda\to\infty}f^{\star}(\lambda/2,\lambda)=\frac{1}{2}\lim
_{\lambda\to\infty}\frac{6+2\lambda}{2\lambda\,+5}=\frac{1}{2}.
\]

\item[iii)] Since
\[
\lim_{\lambda\to\infty}(\mathcal{M}^{\star}(\lambda)-\lambda/2)=\frac{1}%
{2}\lim_{\lambda\rightarrow\infty}\frac{1-e^{\lambda}+\lambda e^{\lambda
}-2\lambda^{2}+(2\lambda-\lambda^{2})\int_{0}^{\lambda}\frac{e^{x}-1}{x}%
dx}{1-e^{\lambda}+2\lambda+(\lambda-1)\int_{0}^{\lambda}\frac{e^{x}-1}{x}dx},
\]
L'Hôpital's rule leads to
\[
\lim_{\lambda\to\infty}(\mathcal{M}^{\star}(\lambda)-\lambda/2)=\frac{1}%
{2}\lim_{\lambda\rightarrow\infty}\frac{-2\lambda-4}{\lambda}=-1
\]
and hence the result.
\end{itemize}
\end{proof}

Note that the value $\mathcal{M}^{\star}(\lambda)\approx\lambda/2-1$ is not
the optimal cutoff value that we are looking for. $\mathcal{M}^{\star}%
(\lambda)$ is the value for which the function $f^{\star}(\cdot,\lambda)$
reaches its maximum, while we are interested in finding $\mathcal{M}(\lambda)$
which is the value for which the function $F^{\mathcal{P}}(\cdot,\lambda)$
reaches its maximum. Of course, both functions are closely related and, as we
will see, so are $\mathcal{M}^{\star}(\lambda)$ and $\mathcal{M}(\lambda)$.

Recall that, in the case when the number of objects is known, the probability
of success following the optimal threshold strategy in the Best-or-Worst
problem is given by
\[
P_{BW}(n)=%
\begin{cases}
\frac{n}{2(n-1)}, & \text{if $n$ is even};\\
\frac{n+1}{2n}, & \text{if $n$ is odd}.
\end{cases}
\]
The next lemma holds.

\begin{lem}
\label{lem:p}
\[
\lim_{\lambda\to\infty}\sum_{k=1}^{\infty}P_{BW}(k)\frac{\lambda^{k}%
}{e^{\lambda}k!}=1/2.
\]

\end{lem}

\begin{proof}
Taking into account the definition of $P_{BW}$ we have the following
decomposition
\[
\sum_{k=1}^{\infty}P_{BW}(k)\frac{\lambda^{k}}{e^{\lambda}k!}=\sum
_{k=1}^{\infty}P_{BW}(2k)\frac{\lambda^{2k}}{e^{\lambda}(2k)!}+\sum
_{k=1}^{\infty}P_{BW}(2k-1)\frac{\lambda^{(2k-1)}}{e^{\lambda}(2k-1)!}%
\]

Now, it can be easily seen that
\begin{align*}
\sum_{k=1}^{\infty}P_{BW}(2k)\frac{\lambda^{2k}}{e^{\lambda}(2k)!}  &
=\sum_{k=1}^{\infty}\frac{2k}{4k-2}\frac{\lambda^{2k}}{e^{\lambda}(2k)!}%
=\frac{\lambda}{2}e^{-\lambda}S(\lambda),\\
\sum_{k=1}^{\infty}P_{BW}(2k-1)\frac{\lambda^{(2k-1)}}{e^{\lambda}(2k-1)!}  &
=\sum_{k=1}^{\infty}\frac{2k}{4k-2}\frac{e^{-\lambda}\lambda^{(2k-1)}%
}{(2k-1)!}=\frac{\sinh(\lambda)}{2\,e^{\lambda}}+\frac{S(\lambda)}%
{2e^{\lambda}},
\end{align*}
where
\[
S(\lambda)=\int_{0}^{\lambda}\frac{\sinh(x)}{x}dx.
\]

Consequently,
\[
\lim_{\lambda\to\infty}\sum_{k=1}^{\infty}P_{BW}(k)\frac{\lambda^{k}%
}{e^{\lambda}k!}=\frac{\lambda}{2}e^{-\lambda}S(\lambda)+\frac{\sinh(\lambda
)}{2e^{\lambda}} + \frac{S(\lambda)}{2e^{\lambda}}%
\]
and the result follows from L'Hôpital's rule.
\end{proof}

Using this result we can compute the asymptotic probability of success using
the optimal threshold strategy in the Poisson case.

\begin{teor}
\label{Treg} With all the previous notation, we have that
\[
\lim_{\lambda\to\infty}F^{\mathcal{P}}(\lambda/2,\lambda)=\lim_{\lambda
\to\infty}F^{\mathcal{P}}\left(  \mathcal{M}(\lambda),\lambda\right)
=\lim_{\lambda\to\infty}P(\lambda)=1/2.
\]

\end{teor}

\begin{proof}
By definition, $P_{BW}(k)$ is the maximum probability of success in the
Best-or-Worst problem with a known number of objects $k>1$. Hence,
\[
\frac{2\mathcal{M}(\lambda)(k-\mathcal{M}(\lambda))}{k(k-1)}=P_{k}%
(\mathcal{M}(\lambda))\leq P_{BW}(k).
\]
Consequently,
\[
P(\lambda)=F^{\mathcal{P}}(\mathcal{M}(\lambda),\lambda)=\sum
_{k=\mathcal{M(\lambda)}+1}^{\infty}\frac{2\mathcal{M}(\lambda)(n-\mathcal{M}%
(\lambda))}{k(k-1)}\frac{\lambda^{k}}{e^{\lambda}k!}\leq\sum_{k=1}^{\infty
}P_{BW}(k)\frac{\lambda^{k}}{e^{\lambda}k!}%
\]
so, if we take upper limits it is clear that $\displaystyle\varlimsup
_{\lambda\to\infty}P(\lambda)\leq1/2$.

On the other hand, recalling Lemma \ref{LEM:f} and Theorem \ref{TEOR:PE}
we have that
\[
P(\lambda)=F^{\mathcal{P}}(\mathcal{M}(\lambda),\lambda)\geq F^{\mathcal{P}%
}(\lambda/2,\lambda)=f^{\star}(\lambda/2,\lambda)-f(\lambda/2,\lambda).
\]
and also that
\[
\lim_{\lambda\to\infty}f^{\star}(\lambda/2,\lambda)-f(\lambda/2,\lambda)=1/2.
\]
Thus, taking lower limits $\displaystyle\varliminf_{\lambda\to\infty}%
P(\lambda)\geq1/2.$

In conclusion, $\displaystyle 1/2 \leq\displaystyle\varliminf_{\lambda
\to\infty}P(\lambda) \leq\varlimsup_{\lambda\to\infty}P(\lambda)\leq1/2$ and
the proof is complete.
\end{proof}

Even if we did not explicitly find an estimation for the optimal cutoff value
$\mathcal{M}(\lambda)$, this theorem satisfactorily solves the problem, since
it implies that $\lambda/2$ is an acceptable estimation for the optimal cutoff
value because it provides the same asymptotic probability of success as the
exact value of $\mathcal{M}(\lambda)$ would do; i.e., $1/2$. Thus, Theorem
\ref{Treg} proves that, if the random number of objects follows a Poisson
distribution $\mathcal{P}(\lambda)$, the optimal strategy consists in
rejecting the first $\lambda/2$ objects and then accept the first nice candidate
after them. With this strategy we will succeed approximately one half of the times.

In addition, as far as we were able to check, $\lfloor\lambda/2-1 \rfloor$
coincides with the exact value of $\mathcal{M}(\lambda)$ for every integer
value of $\lambda>1$. Hence, we propose the following conjecture.

\begin{con}
$\mathcal{M}(\lambda) \approx\mathcal{M}^{\star}(\lambda)\approx\lambda/2-1$.
\end{con}

\section{Concluding remarks}

In the table below we compare the optimal cutoff value $\mathcal{M}$ and the
asymptotic probability of success $P$ in the classical, the Best-or-Worst and
the Postdoc problems and in the different variants studied in this paper for
the number of objects $X$. Recall the relationship between the Best-or-Worst
problem and the Postdoc problem that was stated in Theorem \ref{Tcon}. All the
constants that appear in the table can be expressed in terms of
$e=2.71828\dots$ and the rumor's constant $\vartheta:=-\frac{1}{2} W(-2e^{-2})
=0.20318\dots$

\begin{center}%
\begin{tabular}
[c]{|c|c|c|c|c|c|c|}\hline
& \multicolumn{2}{|c|}{Classic} & \multicolumn{2}{|c|}{Best-or-Worst} &
\multicolumn{2}{|c|}{Postdoc}\\\hline\hline
$X$ & $\mathcal{M}$ & $P$ & $\mathcal{M}$ & $P$ & $\mathcal{M}$ & $P$\\\hline
$X=n$ & $ne^{-1}$ & $e^{-1}$ & $n/2$ & $1/2$ & $n/2$ & $1/4$\\\hline
$X\sim\mathcal{U}[1,n]$ & $ne^{-2}$ & $2e^{-2}$ & $n\vartheta$ &
$2(\vartheta-\vartheta^{2})$ & $n\vartheta$ & $(\vartheta-\vartheta^{2}%
)$\\\hline
$X\sim\mathcal{P}(\lambda)$ & $\lambda/e$ & $e^{-1}$ & $\lambda/2$ & $1/2$ &
$\lambda/2$ & $1/4$\\\hline
\end{tabular}
\end{center}

We close the paper with an intriguing remark pointed out by Havil \cite{intri}
relating the convergents for the continued fraction of $e^{-1}$ and the
optimal cutoff value $r(n)$ in the secretary problem with a known number of
candidates $n$. In fact, the convergents for the continued fraction of
$e^{-1}$ (see sequences A007676 and A007677 in the OEIS) are given by
\[
0,\frac{1}{2},\frac{1}{3},\frac{3}{8},\frac{4}{11}, \frac{7}{19},\frac{32}%
{87}, \frac{39}{106},\frac{71}{193}, \frac{465}{1264},\frac{536}{1457},
\frac{1001}{2721},\dots
\]
which exactly coincide with the fractions of the form $r(n)/n$.

Now, if we focus on the Best-or-Worst problem when the number of candidates
follows a discrete Uniform distribution $\mathcal{U}[1,n]$ the same relation
arises considering $\vartheta$ instead of $e^{-1}$. The convergents for the
continued fraction of $\vartheta$ are
\[
0,\frac{1}{4},\frac{1}{5}, \frac{12}{59},\frac{13}{64}, \frac{38}{187}%
,\frac{51}{251}, \frac{1262}{6211}, \frac{1313}{6462}, \frac{11766}{57907},
\frac{13079}{64369}, \frac{64082}{315383},\dots
\]
which, as far as our computation capabilities allowed us to check, also
coincide with fractions of the form $r(n)/n$, with $r(n)$ being the optimal
cutoff value for the considered problem. However, as Havil's remark, this
relation remains an open problem.

\section*{Acknowledgment}

We wish to thank Prof. Dr. F. Thomas Bruss as well as Prof. Dr. Krzysztof Szajowski for their very useful comments that provided relevant references and increased the interest of the paper.

\end{document}